\documentclass[a4paper,12pt]{amsart}

\usepackage{amsfonts}
\usepackage{amsmath}
\usepackage{amssymb}

\usepackage{mathrsfs}
\usepackage{hyperref}

\numberwithin{equation}{section}
\theoremstyle{plain}
\newtheorem{thm}{Theorem}[section]
\newtheorem{prop}[thm]{Proposition}
\newtheorem{cor}[thm]{Corollary}

\theoremstyle{definition}
\newtheorem{defi}[thm]{Definition}

\def\Sp{{\mathcal H}}
\def\U{\mathcal U}
\def\V{\mathcal V}
\def\ind{\mathbb N}
\def\S{\Lambda}
\def\Dom{\mathrm{Dom}}
\def\L{\mathcal L}
\def\Ls{{\mathcal L}^{\ast}}
\def\C{\mathcal C}
\def\F{\mathcal F}

\def\s{\star}
\def\D{\mathcal D}
\def\Y{\mathcal Y}

\begin{document}

\title[Convolution generated by Riesz bases]
{Convolution, Fourier analysis, and distributions generated by Riesz bases}

\author[Michael Ruzhansky]{Michael Ruzhansky}
\address{
  Michael Ruzhansky:
  \endgraf
  Department of Mathematics
  \endgraf
  Imperial College London
  \endgraf
  180 Queen's Gate, London, SW7 2AZ
  \endgraf
  United Kingdom
  \endgraf
  {\it E-mail address} {\rm m.ruzhansky@imperial.ac.uk}
  }
\author[Niyaz Tokmagambetov]{Niyaz Tokmagambetov}
\address{
  Niyaz Tokmagambetov:
  \endgraf
    al--Farabi Kazakh National University
  \endgraf
  71 al--Farabi ave., Almaty, 050040
  \endgraf
  Kazakhstan,
  \endgraf
   and
  \endgraf
    Department of Mathematics
  \endgraf
  Imperial College London
  \endgraf
  180 Queen's Gate, London, SW7 2AZ
  \endgraf
  United Kingdom
  \endgraf
  {\it E-mail address} {\rm n.tokmagambetov@imperial.ac.uk}
 }

\thanks{The authors were supported in parts by the EPSRC
grants EP/K039407/1 and EP/R003025/1, and by the Leverhulme Grants RPG-2014-02 and RPG-2017-151,
as well as by the MESRK grant 0773/GF4. No new data was collected or generated during the course of research.}

\date{\today}

\subjclass{42A85, 44A35.} \keywords{convolution, basis, biorthogonal system, Fourier analysis, Hilbert space}

\begin{abstract}
In this note we discuss notions of convolutions generated by biorthogonal systems of elements of a Hilbert space.  We develop the associated biorthogonal Fourier analysis and the theory of distributions, discuss properties of convolutions and give a number of examples.
\end{abstract}

\maketitle

\section{Introduction}


In this work we introduce a notion of a convolution generated by systems of elements of a Hilbert space $\Sp$ forming a Riesz basis.

Such collections often arise as systems of eigenfunctions of densely defined non-self-asjoint operators acting on $\Sp$, and a suitable notion of convolution also leads to the development of the associated Fourier analysis. In the case of the eigenfunctions having no zeros the corresponding global theory of pseudo-differential operators has been recently developed in
\cite{RT16}. The assumption on eigenfunctions having no zeros has been subsequently removed in \cite{RT16a}, and some applications of such analysis to the wave equation for the Landau Hamiltonian were carried out in \cite{RT16b}, as well as for general operators with discrete spectrum in \cite{RT17c}. The analysis in these papers relied on the spectral properties of a fixed operator acting in $\Sp=L^{2}(M)$ for a smooth manifold $M$ with or without boundary.

In this note we aim at discussing an abstract point of view on convolutions when one is given only a Riesz basis in a Hilbert space, without making additional assumptions on an operator for which it may be a basis of eigenfunctions. Such an abstract point of view has a number of advantages, for example, the questions of whether the basis elements (for example in $\Sp=L^{2}(M)$) have zeros at some points, become irrelevant.

More specifically, let $\Sp$ be a separable Hilbert space, and denote by
$$
\U:=\{u_{\xi}|\,\, u_{\xi}\in\Sp\}_{\xi\in\ind}
$$
and
$$
\V:=\{v_{\xi}|\,\, v_{\xi}\in\Sp\}_{\xi\in\ind}
$$
collections of elements of $\Sp$ parametrised by a discrete set $\ind$.
We assume that the system $\U$ is a Riesz basis of the space $\Sp$ and the system $\V$ is biorthogonal to $\U$ in $\Sp$, i.e. we have the property that
$$
(u_{\xi},v_{\eta})_{\Sp}=\delta_{\xi\eta},
$$
where $\delta_{\xi\eta}$ is the Kronecker delta, equal to $1$ for $\xi=\eta$, and to $0$ otherwise.
Then from the classical Bari's work \cite{bari} (see also Gelfand \cite{Gel51}) it follows that the system $\V$ is also basis in $\Sp$. The Riesz basis is characterised by the property that it is the image of an orthonormal basis in $\Sp$ under a linear invertible transformation. However, since our aim is to subsequently extend the present constructions in the future beyond the Riesz basis setting we will try not to make explicit use of this property. The results of this paper have been announced in \cite{KRT17}.

The setting of Riesz bases has numerous applications to different problems, see e.g.
\cite{Chr01, Chr16}, and in different settings and modifications, see e.g.  \cite{BT16, GM06, GP11}, to mention only very few. Decomposition systems of different types and the subsequent function spaces is also an active area of research, see e.g. \cite{K03,BH06,B07,GJN17}.

\smallskip
In this paper we define
$\U$-- and $\V$--convolutions in the following form:
\begin{equation}\label{EQ:c1}
f\s_{\U}g:= \sum_{\xi\in\ind}(f, v_\xi) (g, v_\xi) u_{\xi}
\end{equation}
and
\begin{equation}\label{EQ:c2}
h\s_{\V}j:= \sum_{\xi\in\ind}(h, u_\xi) (j, u_\xi) v_{\xi}
\end{equation}
for appropriate elements $f,g,h,j\in \Sp$. These convolutions are clearly commutative and associative, and have a number of properties expected from convolutions, most importantly, they are mapped to the product by the naturally defined Fourier transforms associated to $\U$ and $\V$.

Without going too much into detail, let us briefly summarise the results of this paper:

\begin{itemize}
\item The naturally defined Fourier transforms in $\Sp$ map convolutions \eqref{EQ:c1} and \eqref{EQ:c2} to the product of Fourier transforms. For example,
defining $\widehat{f}(\xi):=(f, v_{\xi})$, we have $\widehat{f\s_{\U} g}=\widehat{f}\,\widehat{g}.$
Moreover, conversely, if a bilinear mapping $K:\Sp\times\Sp\to\Sp$ satisfies $\widehat{K(f,g)}=\widehat{f}\,\widehat{g}$, it must be  given by \eqref{EQ:c1}.
\item Although the bases $\U$ and $\V$ do not have to be orthogonal, there is a Hilbert space $l^{2}_{\U}$ such that we have the Plancherel identity $(f,g)_{\Sp}=(\widehat{f},\widehat{g})_{l^{2}_{\U}}$.
\item We discuss more general families of spaces $l^{p}_{\U}$, $1\leq p\leq\infty$, on the Fourier transform side giving rise to further Fourier analysis in $\Sp$. Namely, these spaces satisfy analogues of the usual duality and interpolation relations, as well as the Hausdorff-Young inequalities with the corresponding family of subspaces of $\Sp$.
\item The developed biorthogonal Fourier analysis can be embedded in an appropriate theory of distributions realised in suitable rigged Hilbert spaces
$(\Phi_{\U}, \Sp, \Phi_{\U}')$ and $(\Phi_{\V}, \Sp, \Phi_{\V}')$, with $\Phi_{\U}:=\C^{\infty}_{\U, \S}$, $\Phi_{\U}':=\mathcal D'_{\V, \S}$ and $\Phi_{\V}:=\C^{\infty}_{\V, \S}$, $\Phi_{\V}':=\mathcal D'_{\U, \S}$, associated to a fixed spectral set $\Lambda$ satisfying certain natural properties. These triples allow us to extend the notions of $\U$-- and $\V$--convolutions and $\U$-- and $\V$--Fourier transforms beyond the Hilbert space $\Sp$.
\item We show how these constructions are related to the spectral decompositions of linear operators in $\Sp$. In particular, we relate the convolutions to the formulae for their resolvents.
\item We discuss several examples and several further possible notions of convolutions.
\end{itemize}

Let us conclude the introduction by giving a concrete example of such convolution also relating it to the spectral analysis.
Let us consider the operator $\L:\Sp\to\Sp$ on the interval $(0, 1)$ given by
$$
\L:= -i\frac{d}{d x},
$$
and let us equip this operator with boundary condition $h y(0)=y(1)$ for some $h>0$.
The operator $\L$ is not self-adjoint on $\Sp=L^{2}(0,1)$ for  $h\not=1$.
The spectral properties of $\L$ have been thoroughly investigated by e.g. Titchmarsh \cite{titc} and Cartwright \cite{cart}. In particular, it is known that the collections
\begin{equation}\label{EQ:u1}
\U=\{u_{j}(x)=h^{x}e^{ 2\pi i x j },\,\, j\in \mathbb{Z}\}
\end{equation}
and
\begin{equation}\label{EQ:v1}
\V=\{v_{j}(x)=h^{-x} e^{2\pi i x j },\,\, j\in \mathbb{Z}\}
\end{equation}
are the systems of eigenfunctions of $\L$ and $\L^{*}$, respectively, and form
Riesz bases in $\Sp=L^{2}(0, 1).$
In this case the abstract definition of convolution above can be shown (see Proposition \ref{prop:Kc}) to yield the concrete expression
$$
(f\star_{\U}
g)(x)=\int^{x}_{0}f(x-t)g(t)dt+\frac{1}{h}\int^{1}_{x}f(1+x-t)g(t)dt,
$$
which coincides with the usual convolution for $h=1$, in which case also $\U=\V$ is an orthonormal basis in $\Sp$. Of course, in this example, the main interest for us is the case $h\not=1$ corresponding to biorthogonal bases $\U$ and $\V$ in \eqref{EQ:u1} and \eqref{EQ:v1}, respectively.

In this paper, to avoid any confusion, we will be using the notation $\mathbb N_{0}=\mathbb N\cup\{0\}.$

\section{Biorthogonal convolutions}

In this section we describe the functional analytic setting for investigating convolutions \eqref{EQ:c1} and \eqref{EQ:c2}.
Let us take biorthogonal systems
$$
\U:=\{u_{\xi}|\,\, u_{\xi}\in\Sp\}_{\xi\in\ind}
$$
and
$$
\V:=\{v_{\xi}|\,\, v_{\xi}\in\Sp\}_{\xi\in\ind}
$$
in a separable Hilbert space $\Sp$, where $\ind$ is a discrete set. We assume that $\U$ (and hence also $\V$) is a Riesz basis in $\Sp$, i.e. any element of $\Sp$ has a unique decomposition with respect of the elements of $\U$. We note that the basis collections are uniformly bounded in $\Sp$.


Before we proceed with describing a version of the biorthogonal Fourier analysis, let us show that expressions in \eqref{EQ:c1} and \eqref{EQ:c2} are usually well-defined.

\begin{prop}\label{PROP:wd}
Let $f\s_{\U}g$ and $h\s_{\V}j$ be defined by \eqref{EQ:c1} and \eqref{EQ:c2}, respectively, that is,
\begin{equation}\label{EQ:c11}
f\s_{\U}g:= \sum_{\xi\in\ind}(f, v_\xi) (g, v_\xi) u_{\xi}
\end{equation}
and
\begin{equation}\label{EQ:c21}
h\s_{\V}j:= \sum_{\xi\in\ind}(h, u_\xi) (j, u_\xi) v_{\xi}.
\end{equation}
Then there exists a constant $M>0$ such that we have
\begin{equation}\label{EQ:est1}
\|f\s_{\U}g\|_{\Sp}\leq M\sup_{\xi\in\ind}\|u_{\xi}\|_{\Sp}\|f\|_{\Sp}\|g\|_{\Sp},\quad
\|h\s_{\V}j\|_{\Sp}\leq M\sup_{\xi\in\ind}\|v_{\xi}\|_{\Sp}\|h\|_{\Sp}\|j\|_{\Sp},
\end{equation}
for all $f,g,h,j\in\Sp.$
\end{prop}
The statement follows from the Cauchy-Schwarz inequality and the following fact:
since systems of  $u_\xi$ and of $v_\xi$ are Riesz bases in $\Sp$, from \cite[Theorem 9]{bari} we have that 
there are constants $a,A,b,B>0$ such that for arbitrary $g\in\Sp$
we obtain
\begin{equation}\label{LEM: FTl2}
a^2\|g\|_{\Sp}^2 \leq \sum_{\xi\in\ind} |(g,v_{\xi})|^2\leq A^2\|g\|_{\Sp}^2 \,\,\,\,\,\,  \hbox{and} \,\,\,\,\,\,  b^2\|g\|_{\Sp}^2 \leq \sum_{\xi\in\ind} |(g,u_{\xi})|^2\leq B^2\|g\|_{\Sp}^2.
\end{equation} 
This amounts to simply stating that the Riesz basis collections form collections of frames in $\Sp$.
From the Riesz basis property it also follows that the families $\U$ and $\V$ are uniformly bounded in $\Sp$, that is,
$$
\sup_{\xi\in\ind}\|u_{\xi}\|_{\Sp}+\sup_{\xi\in\ind}\|v_{\xi}\|_{\Sp}<\infty.
$$

Let us introduce $\U$-- and $\V$--Fourier transforms by formulas
\begin{equation}\label{EQ: FT_u}
\F_{\U}(f)(\xi):=(f, v_{\xi})=:\widehat{f}(\xi)
\end{equation}
and
\begin{equation}\label{EQ: FT_v}
\F_{\V}(g)(\xi):=(g, u_{\xi})=:\widehat{g}_{\ast}(\xi),
\end{equation}
respectively, for all $f, g\in\Sp$ and for each $\xi\in\ind$. Here $\widehat{g}_{\ast}$ stands for the $\V$--Fourier transform of the function $g$. Indeed, in general $\widehat{g}_{\ast}\neq\widehat{g}$.
Their inverses are given by
\begin{equation}\label{EQ: FT_ui}
(\F_{\U}^{-1}a)(x):=\sum_{\xi\in\ind} a(\xi)u_{\xi}
\end{equation}
and
\begin{equation}\label{EQ: FT_vi}
(\F_{\V}^{-1}a)(x):=\sum_{\xi\in\ind} a(\xi)v_{\xi}.
\end{equation}

The Fourier transforms defined in \eqref{EQ: FT_u} and \eqref{EQ: FT_v} are the analysis operators, and, the inverse transforms \eqref{EQ: FT_ui} and \eqref{EQ: FT_vi} are the corresponding synthesis operators, see e.g. \cite{K03}. For more information, see e.g. \cite{BH06, B07, GJN17} and references therein. 

There is a straightforward relation between $\U$- and $\V$-convolutions, and the Fourier transforms:

\begin{thm}\label{PR: ConvProp}
For arbitrary $f, g, h, j\in\Sp$ we have
$$
\widehat{f\s_{\U} g}=\widehat{f}\,\widehat{g}, \,\,\,\, \widehat{h\s_{\V} j}_{\ast}=\widehat{h}_{\ast}\,\widehat{j}_{\ast}.
$$

Therefore, the convolutions are commutative and associative.

Let $K:\Sp\times\Sp\to\Sp$ be a bilinear mapping. If for all $f,g\in\Sp$, the form $K(f,g)$ satisfies the property
\begin{equation}\label{EQ: Bilinear op-2}
\widehat{K(f, g)}=\widehat{f} \,\widehat{g}
\end{equation}
then $K$ is the $\U$--convolution, i.e. $K(f,g)=f*_{\U}g$.

Similarly, if $K(f,g)$ satisfies the property
\begin{equation}\label{EQ: Bilinear op-22}
\widehat{K(f, g)}_{*}=\widehat{f}_{*} \,\widehat{g}_{*}
\end{equation}
then $K$ is the $\V$--convolution, i.e. $K(f,g)=f*_{\V}g$.
\end{thm}
\begin{proof}
Direct calculations yield
\begin{align*}
\mathcal F_{\U}(f\s_{\U}
g)(\xi)&=(\sum_{\eta\in\ind}
\widehat{f}(\eta)\widehat{g}(\eta)u_{\eta}, v_{\xi})
\\
&=\sum_{\eta\in\ind}
\widehat{f}(\eta)\widehat{g}(\eta)(u_{\eta}, v_{\xi})
\\
&=\widehat{f}(\xi)\widehat{g}(\xi).
\end{align*}
Commutativity follows from the bijectivity of the $\U$--Fourier transform, also implying the associativity. This can be also seen from the definition:
\begin{align*}
((f\s_{\U} g) \s_{\U} h) & = \sum_{\xi\in\ind}(\sum_{\eta\in\ind}
\widehat{f}(\eta)\widehat{g}(\eta)u_{\eta}, v_{\xi})\widehat{h}(\xi)u_{\xi}
\\
&=\sum_{\xi\in\ind}\widehat{f}(\xi)\widehat{g}(\xi)\widehat{h}(\xi)u_{\xi}
\\
&=\sum_{\xi\in\ind}\widehat{f}(\xi)\left[\sum_{\eta\in\ind}\widehat{g}(\eta)\widehat{h}(\eta)(u_{\eta}, v_{\xi})\right]u_{\xi}
\\
&=\sum_{\xi\in\ind}\widehat{f}(\xi)(\sum_{\eta\in\ind}
\widehat{g}(\eta)\widehat{h}(\eta)u_{\eta}, v_{\xi})u_{\xi}
\\
&=(f\s_{\U} (g \s_{\U} h)).
\end{align*}
Next, it is enough to prove that $K$ is the $\U$-convolution under the assumption \eqref{EQ: Bilinear op-2}. The similar property for $\V$-convolutions under assumption \eqref{EQ: Bilinear op-22} follows by simply replacing $\U$ by $\V$ in the part concerning $\U$-convolutions.

Since for arbitrary $f,g\in\Sp$ and for $K(f,g)\in\Sp$ the property \eqref{EQ: Bilinear op-2} is valid then we can regain $K(f,g)$ from the inverse $\U$--Fourier formula
$$
K(f,g)=\sum_{\xi\in\ind}\widehat{K(f, g)}(\xi)u_{\xi}=\sum_{\xi\in\ind}\widehat{f}(\xi) \,\widehat{g}(\xi) u_{\xi}.
$$
The last expression defines the $\U$--convolution.
\end{proof}

\section{Biorthogonal Fourier analysis}
\label{SEC:FA}

From \eqref{LEM: FTl2} we can conclude that the $\U$-- and $\V$--Fourier coefficients of the elements of $\Sp$ belong to the space of square-summable sequences $l^{2}(\ind)$.
However, we note that the Plancherel identity is also valid for suitably
defined $l^2$-spaces of Fourier coefficients, see \cite[Proposition 6.1]{RT16}.
We explain it now in the present setting.

Indeed, the frame property in \eqref{LEM: FTl2} can be improved to the exact Plancherel formula with a suitable choice of norms.

\subsection{Plancherel formula}

Let us denote by $$l^{2}_{\U}=l^2(\U)$$
the linear space of complex-valued functions $a$
on $\ind$ such that $\mathcal F^{-1}_{\U}a\in
\Sp$, i.e. if there exists $f\in \Sp$ such that $\mathcal F_{\U}f=a$.
Then the space of sequences $l^{2}_{\U}$ is a
Hilbert space with the inner product
\begin{equation}\label{EQ: InnerProd SpSeq-s}
(a,\ b)_{l^{2}_{\U}}:=\sum_{\xi\in\ind}a(\xi)\ \overline{(\mathcal F_{\V}\circ\mathcal F^{-1}_{\U}b)(\xi)},
\end{equation}
for arbitrary $a,\,b\in l^{2}_{\U}$.
The reason for this choice of the definition is the following formal calculation:
\begin{align}\label{EQ:PL-prelim} \nonumber
(a,\ b)_{l^{2}_{\U}}&
=\sum_{\xi\in\ind}a(\xi)\ \overline{(\mathcal F_{\V}\circ\mathcal F^{-1}_{\U}b)(\xi)}\\ \nonumber
&=\sum\limits_{\xi\in\ind
}a(\xi)\overline{\left(\mathcal F^{-1}_{\U}b, u_{\xi}\right)}\\ \nonumber
&=\left(\left[\sum\limits_{\xi\in\ind}a(\xi)u_{\xi}\right], \mathcal F^{-1}_{\U}b\right)\\ \nonumber
&=(\mathcal F^{-1}_{\U}a,\,\mathcal F^{-1}_{\U}b),
\end{align}
which implies the Hilbert space properties of the space of sequences
$l^{2}_{\U}$. The norm of $l^{2}_{\U}$ is then given by the
formula
\begin{equation}\label{EQ:l2norm}
\|a\|_{l^{2}_{\U}}=\left(\sum_{\xi\in\ind}a(\xi)\
\overline{(\mathcal F_{\V}\circ\mathcal F^{-1}_{\U}a)(\xi)}\right)^{1/2}, \quad \textrm{ for all } \; a\in l^{2}_{\U}.
\end{equation}
We note that individual terms in this sum may be complex-valued but the
whole sum is real and non-negative.

Analogously, we introduce the
Hilbert space $$l^{2}_{\V}=l^{2}(\V)$$
as the space of functions $a$ on $\ind$
such that $\mathcal F^{-1}_{\V}a\in \Sp$,
with the inner product
\begin{equation}
\label{EQ: InnerProd SpSeq-s_2} (a,\ b)_{l^{2}_{\V}}:=\sum_{\xi\in\ind}a(\xi)\ \overline{(\mathcal
F_{\U}\circ\mathcal F^{-1}_{\V}b)(\xi)},
\end{equation}
for arbitrary $a,\,b\in l^{2}_{\V}$. The norm of
$l^{2}_{\V}$ is given by the formula
$$
\|a\|_{l^{2}_{\V}}=\left(\sum_{\xi\in\ind}a(\xi)\
\overline{(\mathcal F_{\U}\circ\mathcal F^{-1}_{\V}a)(\xi)}\right)^{1/2}
$$
for all $a\in l^{2}_{\V}$. The spaces of sequences
$l^{2}_{\U}$ and
$l^{2}_{\V}$ are thus generated by biorthogonal systems
$\{u_{\xi}\}_{\xi\in\ind}$ and $\{v_{\xi}\}_{\xi\in\ind}$.

Since Riesz bases are equivalent to an orthonormal basis by an invertible linear transformation, we have the equality between the spaces $l^{2}_{\U}=l^{2}_{\V}=l^{2}(\mathbb N)$ as sets; of course the special choice of their norms is the important ingredient in their definition.

Indeed, the reason for their definition in the above forms becomes clear again
in view of the following Plancherel identity:

\begin{thm} {\rm(Plancherel's identity)}\label{PlanchId}
If $f,\,g\in \Sp$ then
$\widehat{f},\,\widehat{g}\in l^{2}_{\U}, \,\,\,
\widehat{f}_{\ast},\, \widehat{g}_{\ast}\in l^{2}_{\V}$, and the inner products {\rm(\ref{EQ: InnerProd SpSeq-s}),
(\ref{EQ: InnerProd SpSeq-s_2})} take the form
$$
(\widehat{f},\ \widehat{g})_{l^{2}_{\U}}=\sum_{\xi\in\ind}\widehat{f}(\xi)\ \overline{\widehat{g}_{\ast}(\xi)}
$$
and
$$
(\widehat{f}_{\ast},\ \widehat{g}_{\ast})_{l^{2}_{\V}}=\sum_{\xi\in\ind}\widehat{f}_{\ast}(\xi)\
\overline{\widehat{g}(\xi)},
$$
respectively. In particular, we have
$$
\overline{(\widehat{f},\ \widehat{g})_{l^{2}_{\U}}}=
(\widehat{g}_{\ast},\ \widehat{f}_{\ast})_{l^{2}_{\V}}.
$$
The Parseval identity takes the form
\begin{equation}\label{Parseval}
(f,g)_{\Sp}=(\widehat{f},\widehat{g})_{l^{2}_{\U}}=\sum_{\xi\in\ind}\widehat{f}(\xi)\ \overline{\widehat{g}_{\ast}(\xi)}.
\end{equation}
Furthermore, for any $f\in \Sp$, we have
$\widehat{f}\in l^{2}_{\U}$, $\widehat{f}_{\ast}\in l^{2}_{\V}$, and
\begin{equation}
\label{Planch} \|f\|_{\Sp}=\|\widehat{f}\|_{l^{2}_{\U}}=\|\widehat{f}_{\ast}\|_{l^{2}_{\V}}.
\end{equation}
\end{thm}

\begin{proof}
By the definition we get
\begin{align*}
(\mathcal F_{\V}\circ\mathcal F^{-1}_{\U}\widehat{g})(\xi)=\left(\mathcal
F_{\V}g\right)(\xi)=\widehat{g}_{\ast}(\xi)
\end{align*}
and
\begin{align*}
(\mathcal F_{\U}\circ\mathcal F^{-1}_{\V}\widehat{g}_{\ast})(\xi)=\left(\mathcal
F_{\U}g\right)(\xi)=\widehat{g}(\xi).
\end{align*}
Hence it follows that
$$
(\widehat{f},\ \widehat{g})_{l^{2}_{\U}}=\sum_{\xi\in\ind}\widehat{f}(\xi)\ \overline{(\mathcal F_{\V}\circ\mathcal F^{-1}_{\U}\widehat{g})(\xi)}=\sum_{\xi\in\ind}\widehat{f}(\xi)\
\overline{\widehat{g}_{\ast}(\xi)}
$$
and
$$
(\widehat{f}_{\ast},\ \widehat{g}_{\ast})_{l^{2}_{\V}}=\sum_{\xi\in\ind}\widehat{f}_{\ast}(\xi)\
\overline{(\mathcal F_{\U}\circ\mathcal F^{-1}_{\V}\widehat{g}_{\ast})(\xi)}=\sum_{\xi\in\ind}\widehat{f}_{\ast}(\xi)\ \overline{\widehat{g}(\xi)}.
$$
To show Parseval's identity \eqref{Parseval}, using these properties and the
biorthogonality of $u_\xi$'s to $v_\eta$'s,
we can write
\begin{multline*}
(f,g)=\left(\sum_{\xi\in\ind}\widehat{f}(\xi)u_{\xi} \ , \
\sum_{\eta\in\ind}\widehat{g}_{\ast}(\eta)v_{\eta}\right)\\
=\sum_{\xi\in\ind}\sum_{\eta\in\ind}\widehat{f}(\xi)\overline{\widehat{g}_{\ast}(\eta)}\left(u_{\xi},
\ v_{\eta}\right)
=\sum_{\xi\in\ind}\widehat{f}(\xi)\overline{\widehat{g}_{\ast}(\xi)}=(\widehat{f},\widehat{g})_{l^{2}_{\U}},
\end{multline*}
proving \eqref{Parseval}.
Taking $f=g$, we get
\begin{equation*}
\|f\|_{\Sp}^{2}=(f,f)=
\sum_{\xi\in\ind}\widehat{f}(\xi)\overline{\widehat{f}_{\ast}(\xi)}=(\widehat{f},\widehat{f})_{l^{2}_{\U}}=\|\widehat{f}\|_{l^{2}_{\U}}^{2},
\end{equation*}
proving the first equality in \eqref{Planch}.
Then, by checking that
\begin{align*}
(f,f)=\overline{(f,f)}&=\sum_{\xi\in\ind}
\overline{\widehat{f}(\xi)}\widehat{f}_{\ast}(\xi)=\sum_{\xi\in\ind}
\widehat{f}_{\ast}(\xi)\overline{\widehat{f}(\xi)}=(\widehat{f}_{\ast},\widehat{f}_{\ast})_{l^{2}_{\V}}
=\|\widehat{f}_{\ast}\|_{l^{2}_{\V}}^{2},
\end{align*}
the proofs of \eqref{Planch} and of Theorem \ref{PlanchId} are complete.
\end{proof}

\subsection{Hausdorff-Young inequality}

Now, we introduce a set of Banach spaces $\{\Sp^{p}\}_{1\leq p\leq\infty}$ with the norms $\|\cdot\|_{p}$ such that
$$
\Sp^{p}\subseteq\Sp
$$
and with the property
\begin{equation}\label{EQ:Holder}
|(x, y)_{\Sp}|\leq\|x\|_{\Sp^{p}} \|y\|_{\Sp^{q}}
\end{equation}
for all $1\leq p\leq\infty$, where $\frac1p+\frac1q=1$.
We assume that $\Sp^{2}=\Sp$, and that
$\Sp^{p}$ are real interpolation properties in the following sense:
$$
(\Sp^{1}, \Sp^{2})_{\theta,p}=\Sp^{p}, \; 0<\theta<1,\; \frac1p=1-\frac{\theta}{2},
$$
and
$$
(\Sp^{1}, \Sp^{2})_{\theta,p}=\Sp^{p}, \; 0<\theta<1,\; \frac1p=\frac{1-\theta}{2}.
$$

We also assume that $\U\subset\Sp^{p}$ and $\V\subset\Sp^{p}$ for all $p\in[1, \infty]$.

If $\Sp=L^{2}(\Omega)$ for some $\Omega$, we could take $\Sp^{p}=L^{2}(\Omega)\cap L^{p}(\Omega)$. If $\Sp=S_{2}(\mathcal K)$ is the Hilbert space of Hilbert-Schmidt operators on a Hilbert space $\mathcal K$, then we can take $\Sp^{p}=S_{2}(\mathcal K)\cap S_{p}(\mathcal K),$ where $S_{p}(\mathcal K)$ stands for the space of $p$-Schatten operators on $\mathcal K.$

Below we introduce the $p$-Lebesgue versions of the spaces of Fourier coefficients.
Here classical $l^p$ spaces on $\ind$ are extended in a way so that we associate them to the given biorthogonal systems.

\begin{defi} Let us define spaces $l^{p}_{\U}=l^{p}(\U)$ as the spaces of all
$a:\ind\to\mathbb C$ such that
\begin{equation}\label{EQ:norm1}
\|a\|_{l^{p}(\U)}:=\left(\sum_{\xi\in\ind}| a(\xi)|^{p}
\|u_{\xi}\|^{2-p}_{\Sp^{\infty}} \right)^{1/p}<\infty,\quad \textrm{ for }\; 1\leq p\leq2,
\end{equation}
and
\begin{equation}\label{EQ:norm2}
\|a\|_{l^{p}(\U)}:=\left(\sum_{\xi\in\ind}| a(\xi)|^{p}
\|v_{\xi}\|^{2-p}_{\Sp^{\infty}} \right)^{1/p}<\infty,\quad \textrm{ for }\; 2\leq p<\infty,
\end{equation}
and, for $p=\infty$,
$$
\|a\|_{l^{\infty}(\U)}:=\sup_{\xi\in\ind}\left( |a(\xi)|\cdot
\|v_{\xi}\|^{-1}_{\Sp^{\infty}}\right)<\infty.
$$
\end{defi}
Here, without loss of generality, we can assume that $u_{\xi}\not=0$ and $v_{\xi}\not=0$ for all $\xi\in\ind$, so that the above spaces are well-defined.

Analogously, we introduce spaces $l^{p}_{\V}=l^{p}(\V)$ as the spaces of
all $b:\ind\to\mathbb C$ such that
$$
\|b\|_{l^{p}(\V)}=\left(\sum_{\xi\in\ind}|
b(\xi)|^{p} \|v_{\xi}\|^{2-p}_{\Sp^{\infty}} \right)^{1/p}<\infty,\quad \textrm{ for }\; 1\leq p\leq2,
$$
$$
\|b\|_{l^{p}(\V)}=\left(\sum_{\xi\in\ind}|
b(\xi)|^{p} \|u_{\xi}\|^{2-p}_{\Sp^{\infty}}
\right)^{1/p}<\infty,\quad \textrm{ for }\; 2\leq p<\infty,
$$
$$
\|b\|_{l^{\infty}(\V)}=\sup_{\xi\in\ind}\left(|b(\xi)|\cdot \|u_{\xi}\|^{-1}_{\Sp^{\infty}}\right).
$$
Now, we recall a theorem on the interpolation of weighted spaces from Bergh and L\"ofstr\"om
\cite[Theorem 5.5.1]{Bergh-Lofstrom:BOOK-Interpolation-spaces}. Then we formulate some basic properties of $l^{p}(\U)$.

\begin{thm}[Interpolation of weighted spaces] \label{TH: IWS}
Let us write
$d\mu_{0}(x)=\omega_{0}(x)d\mu(x),$
$d\mu_{1}(x)=\omega_{1}(x)d\mu(x),$ and write
$L^{p}(\omega)=L^{p}(\omega d\mu)$ for the weight $\omega$.
Suppose that $0<p_{0}, p_{1}<\infty$. Then
$$
(L^{p_{0}}(\omega_{0}), L^{p_{1}}(\omega_{1}))_{\theta,
p}=L^{p}(\omega),
$$
where $0<\theta<1$, $\frac{1}{p}=\frac{1-\theta}{p_{0}}+\frac{\theta}{p_{1}}$, and
$\omega=\omega_{0}^{\frac{p(1-\theta)}{p_{0}}}\omega_{1}^{\frac{p\theta}{p_{1}}}$.
\end{thm}

From this we obtain the following property:

\begin{cor}[Interpolation of $l^{p}(\U)$ and $l^{p}(\V)$]
\label{COR:Interp}
For $1\leq p\leq2$, we obtain
$$
(l^{1}(\U), l^{2}(\U))_{\theta,p}=l^{p}(\U),
$$
$$
(l^{1}(\V), l^{2}(\V))_{\theta,p}=l^{p}(\V),
$$
where $0<\theta<1$ and $p=\frac{2}{2-\theta}$.
\end{cor}

Using Theorem \ref{TH: IWS} and Corollary \ref{COR:Interp} we get the following Hausdorff-Young inequality.

\begin{thm}[Hausdorff-Young inequality] \label{TH: HY}
Assume that $1\leq
p\leq2$ and $\frac{1}{p}+\frac{1}{p'}=1$. Then there exists a constant $C_{p}\geq 1$ such that
\begin{equation}\label{EQ:HY}
\|\widehat{f}\|_{l^{p'}(\U)}\leq C_{p}\|f\|_{\Sp^{p}}\quad \textrm{ and }\quad \|\mathcal F_{\U}^{-1}a\|_{\Sp^{p'}}\leq C_{p}\|a\|_{l^{p}(\U)}
\end{equation}
for all $f\in \Sp^{p}$
and $a\in l^{p}(\U)$. Similarly, for all $b\in l^{p}(\V)$ we obtain
\begin{equation}\label{EQ:HYast}
\|\widehat{f}_*\|_{l^{p'}(\V)}\leq C_{p}\|f\|_{\Sp^{p}}\quad \textrm{ and }
\quad \|\mathcal F_{\V}^{-1}b\|_{\Sp^{p'}}\leq C_{p}\|b\|_{l^{p}(\V)}.
\end{equation}
\end{thm}

\begin{proof}
It is sufficient to prove only \eqref{EQ:HY} since \eqref{EQ:HYast} is similar. Note that
\eqref{EQ:HY} would follow from the $\Sp^{1}\rightarrow
l^{\infty}(\U)$ and $l^{1}(\U)\rightarrow \Sp^{\infty}$ boundedness
in view of the Plancherel identity in Theorem \ref{PlanchId}
by interpolation, see e.g. Bergh and L\"ofstr\"om
\cite[Corollary 5.5.4]{Bergh-Lofstrom:BOOK-Interpolation-spaces}.

Thereby, we can put $p=1$. Then from \eqref{EQ:Holder} we have
$$
|\widehat{f}(\xi)|\leq\|{v_{\xi}}\|_{\Sp^{\infty}}\|f\|_{\Sp^{1}},
$$
and hence
$$
\|\widehat{f}\|_{l^{\infty}(\U)}=\sup_{\xi\in\ind}|\widehat{f}(\xi)|
\|v_{\xi}\|^{-1}_{\Sp^{\infty}}\leq\|f\|_{\Sp^{1}}.
$$
The last estimate gives the first inequality in \eqref{EQ:HY} for $p=1$.
For the second inequality, using
$$(\mathcal F_{\U}^{-1}a)=\sum\limits_{\xi\in\ind}a(\xi)u_{\xi}$$
we obtain
$$
\|\mathcal F_{\U}^{-1}a\|_{\Sp^{\infty}}\leq\sum\limits_{\xi\in\ind}|a(\xi)|\|u_{\xi}\|_{\Sp^{\infty}}
=\|a\|_{l^{1}(\U)},
$$
in view of the definition of $l^{1}(\U)$,
which gives \eqref{EQ:HY} in the case $p=1$.
The proof is complete.
\end{proof}

Let us establish the duality between spaces $l^{p}(\U)$ and
$l^{q}(\V)$:

\begin{thm}[Duality of $l^{p}(\U)$ and $l^{q}(\V)$] \label{TH:Duality lp}
Let $1\leq p<\infty$ and
$\frac{1}{p}+\frac{1}{q}=1$. Then
$$\left(l^{p}(\U)\right)'=l^{q}(\V) \quad \textrm{ and }\quad \left(l^{p}(\V)\right)'=l^{q}(\U).$$
\end{thm}
\begin{proof}
The proof is standard. Meanwhile, we provide several details for clarity.
The duality is given by
$$
(\sigma_{1}, \sigma_{2})=\sum\limits_{\xi\in\ind
}\sigma_{1}(\xi){\sigma_{2}(\xi)}
$$
for $\sigma_{1}\in l^{p}(\U)$ and $\sigma_{2}\in l^{q}(\V)$.
Let $1<p\leq2$. Then, if $\sigma_{1}\in l^{p}(\U)$
and $\sigma_{2}\in l^{q}(\V)$, we obtain
\begin{align*}
|(\sigma_{1}, \sigma_{2})|&= \left|\sum_{\xi\in\ind}
\sigma_{1}(\xi)\sigma_{2}(\xi)\right|\\
&=\left|\sum_{\xi\in\ind}
\sigma_{1}(\xi)\|u_{\xi}\|_{\Sp^{\infty}}^{\frac{2}{p}-1}\|u_{\xi}\|_{\Sp^{\infty}}^{-(\frac{2}{p}-1)}\sigma_{2}(\xi)\right|\\
&\leq\left(\sum_{\xi\in\ind}
|\sigma_{1}(\xi)|^{p}\|u_{\xi}\|_{\Sp^{\infty}}^{p(\frac{2}{p}-1)}\right)^{p}\left(\sum_{\xi\in\ind}
|\sigma_{2}(\xi)|^{q}\|u_{\xi}\|_{\Sp^{\infty}}^{-q(\frac{2}{p}-1)}\right)^{\frac{1}{q}}\\
&=\|\sigma_{1}\|_{l^{p}(\U)}\|\sigma_{2}\|_{l^{q}(\V)},
\end{align*}
where that $2\leq q<\infty$ and that
$\frac{2}{p}-1=1-\frac{2}{q}$ were used (last line).
Now, let $2<p<\infty$. If $\sigma_{1}\in l^{p}(\U)$
and $\sigma_{2}\in l^{q}(\V)$, we get
\begin{align*}
|(\sigma_{1}, \sigma_{2})|&= \left|\sum_{\xi\in\ind}
\sigma_{1}(\xi)\sigma_{2}(\xi)\right|\\
&=\left|\sum_{\xi\in\ind}
\sigma_{1}(\xi)\|v_{\xi}\|_{\Sp^{\infty}}^{\frac{2}{p}-1}\|v_{\xi}\|_{\Sp^{\infty}}^{-(\frac{2}{p}-1)}\sigma_{2}(\xi)\right|\\
&\leq\left(\sum_{\xi\in\ind}
|\sigma_{1}(\xi)|^{p}\|v_{\xi}\|_{\Sp^{\infty}}^{p(\frac{2}{p}-1)}\right)^{p}\left(\sum_{\xi\in\ind}
|\sigma_{2}(\xi)|^{q}\|v_{\xi}\|_{\Sp^{\infty}}^{-q(\frac{2}{p}-1)}\right)^{\frac{1}{q}}\\
&=\|\sigma_{1}\|_{l^{p}(\U)}\|\sigma_{2}\|_{l^{q}(\V)}.
\end{align*}

Put $p=1$. Then we have
\begin{align*}
|(\sigma_{1}, \sigma_{2})|&= \left|\sum_{\xi\in\ind}
\sigma_{1}(\xi)\sigma_{2}(\xi)\right|\\
&=\left|\sum_{\xi\in\ind}
\sigma_{1}(\xi)\|u_{\xi}\|_{\Sp^{\infty}}\|u_{\xi}\|_{\Sp^{\infty}}^{-1}\sigma_{2}(\xi)\right|\\
&\leq\left(\sum_{\xi\in\ind}
|\sigma_{1}(\xi)|\,\|u_{\xi}\|_{\Sp^{\infty}}\right)\sup_{\xi\in\ind}|\sigma_{2}(\xi)|\,\|u_{\xi}\|^{-1}_{\Sp^{\infty}}\\
&=\|\sigma_{1}\|_{l^{1}(\U)}\|\sigma_{2}\|_{l^{\infty}(\V)}.
\end{align*}
The adjoint space cases could be proven in a similar way.
\end{proof}

\section{Rigged Hilbert spaces}

In this section we will investigate a rigged structure of the Hilbert space $\Sp$. Especially, we will construct a (Gelfand) triple $(\Phi, \Sp, \Phi')$ with the inclusion property
$$
\Phi\subset\Sp\subset\Phi',
$$
where a role of $\Phi$ will be played by the so-called `spaces of test functions' $\C^{\infty}_{\U, \S}$ and  $\C^{\infty}_{\V, \S}$ generated by the systems $\U$ and $\V$, respectively, and by some sequence $\S$ of complex numbers. For this aim, let us fix some sequence $\S:=\{\lambda_{\xi}\}_{\xi\in\ind}$ of complex numbers such that the series
\begin{equation}\label{EQ:asl}
\sum_{\xi\in\ind}(1+|\lambda_{\xi}|)^{-s_{0}}<\infty,
\end{equation}
converges for some $s_{0}>0$. Indeed, we build two triples, namely,  $(\Phi_{\U}, \Sp, \Phi_{\U}')$ and $(\Phi_{\V}, \Sp, \Phi_{\V}')$, with $\Phi_{\U}:=\C^{\infty}_{\U, \S}$, $\Phi_{\U}':=\mathcal D'_{\V, \S}$ and $\Phi_{\V}:=\C^{\infty}_{\V, \S}$, $\Phi_{\V}':=\mathcal D'_{\U, \S}$. These triples allow us to extend the notions of $\U$-- and $\V$--convolutions and $\U$-- and $\V$--Fourier transforms outside of the Hilbert space $\Sp$.

\begin{defi}\label{DEF:L}
We associate to the pair $(\U, \S)$ a linear operator $\L:\Sp\to\Sp$ by the formula
\begin{equation}\label{EQ-1.1}
\L f:=\sum_{\xi\in\ind}\lambda_{\xi}(f, v_{\xi}) u_{\xi},
\end{equation}
for those $f\in\Sp$ for which the series converges in $\Sp$.
Then $\L$ is densely defined since $\L u_{\xi}=\lambda_{\xi}u_{\xi}$ for all $\xi\in\ind$, and $\U$ is a basis in $\Sp$.
We denote by $\Dom(\L)$ the domain of the operator $\L$, so that we have
${\rm Span}\,(\U)\subset \Dom(\L)\subset\Sp$. We call $\L$ to be the operator associated to the pair $(\U, \S)$. Operators defined as in \eqref{EQ-1.1} have been also studied in \cite{BIT14}.
\end{defi}

We note that this construction goes in the opposite direction to the investigations devoted to the development of the global theory of pseudo-differential operators associated to a fixed operator, as in the papers \cite{DR14, DRT16, RT16, RT16a, RT16b}, where one is given an operator $\L$ acting in $\Sp$ with the system of eigenfunctions $\U$ and eigenvalues $\S$. In this case we could `control' only one parameter, i.e. the operator $\L$. In the present (more abstract) point of view we have two parameters to control: the system $\U$ and the sequence of numbers $\S$.

\smallskip
In a similar way to Definition \ref{DEF:L}, we define the operator $\Ls:\Sp\to\Sp$ by
$$
\Ls g:=\sum_{\xi\in\ind}\overline{\lambda_{\xi}}(g, u_{\xi}) v_{\xi},
$$
for those $g\in\Sp$ for which it makes sense. Then $\Ls$ is densely defined since $\Ls v_{\xi}=\overline{\lambda_{\xi}}v_{\xi}$ and $\V$ is a basis in $\Sp$, and
${\rm Span}\,(\V)\subset \Dom(\Ls)\subset\Sp$.
One readily checks that we have
$$(\L f,g)_{\Sp}=(f,\L^{*}g)_{\Sp}=\sum_{\xi\in\ind}\lambda_{\xi} (f,v_{\xi})(g,u_{\xi})
$$
on their domains.

\smallskip
We can now define the following notions:
\begin{itemize}
\item[(i)] the spaces of $(\U, \S)$-- and $(\V, \S)$--test functions are defined by
$$
\C^{\infty}_{\U, \S}:=\bigcap_{k\in\mathbb N_{0}}\C^{k}_{\U, \S},
$$
where
$$
\C^{k}_{\U, \S}:=\{\phi\in\Sp: \,\, |(\phi, v_{\xi})|\leq C (1+|\lambda_{\xi}|)^{-k} \,\,\, \hbox{for some constant} \,\, C \,\, \hbox{for all} \,\, \xi\in\ind\},
$$
and
$$
\C^{\infty}_{\V, \S}:=\bigcap_{k\in\mathbb N_{0}}\C^{k}_{\V, \S},
$$
where
$$
\C^{k}_{\V, \S}:=\{\psi\in\Sp: \,\, |(\psi, u_{\xi})|\leq C (1+|\lambda_{\xi}|)^{-k} \,\,\, \hbox{for some constant} \,\, C \,\, \hbox{for all} \,\, \xi\in\ind\}.
$$
%
%
%
The topology of these spaces is defined by a natural choice of seminorms.
We can define spaces of $(\U, \S)$-- and $(\V, \S)$--distributions by
$ \D'_{\U,\S}:=(\C^{\infty}_{\V, \S})'$ and $ \D'_{\V,\S}:=(\C^{\infty}_{\U, \S})'$, as spaces of linear continuous functionals on $\C^{\infty}_{\V, \S}$ and $\C^{\infty}_{\U, \S}$, respectively.
We follow the conventions of rigged Hilbert spaces to denote this duality by
\begin{equation}\label{EQ:dual}
\langle u,\phi\rangle_{\D'_{\U,\S}, \C^{\infty}_{\V, \S}}=(u,\phi)_{\Sp},
\end{equation}
extending the inner product on $\Sp$ for $u,\phi\in\Sp,$
and similarly for the pair $ \D'_{\V,\S}:=(\C^{\infty}_{\U, \S})'$.

\item[(ii)] the $\U$-- and $\V$--Fourier transforms
\begin{equation*}\label{EQ: F_u}
\F_{\U}(\phi)(\xi):=(\phi, v_{\xi})=:\widehat{\phi}(\xi)
\end{equation*}
and
\begin{equation*}\label{EQ: F_v}
\F_{\V}(\psi)(\xi):=(\psi, u_{\xi})=:\widehat{\psi}_{\ast}(\xi),
\end{equation*}
respectively, for arbitrary $\phi\in\C^{\infty}_{\U, \S}$, $\psi\in\C^{\infty}_{\V, \S}$ and for all $\xi\in\ind$, and hence by duality, these extend to $\D'_{\U,\S}$ and $ \D'_{\V,\S}$, respectively. Here we have
\begin{equation}\label{EQ: F_vd}
\langle\F_{\U}(w), a\rangle=\langle w, \F_{\V}^{-1}(a)\rangle,\quad
w\in  \D'_{\U,\S},\; a\in \mathcal S(\ind),
\end{equation}
where the space $\mathcal S(\ind)$ is defined in \eqref{EQ:ssp}.
Indeed, for $w\in\Sp$ we can calculate
\begin{multline*}\label{EQ:}
\langle \F_{\U}(w), a\rangle=(\widehat{w},a)_{\ell^{2}(\ind)}
=\sum_{\xi\in\ind}(w,v_{\xi}) \overline{a(\xi)}\\ =
\left(w, \sum_{\xi\in\ind} a(\xi) v_{\xi}\right)=
\left(w,\F_{\V}^{-1} a\right)=
\langle w,\F_{\V}^{-1}a\rangle,
\end{multline*}
justifying definition \eqref{EQ: F_vd}.
Similarly, we define
\begin{equation}\label{EQ: F_vdv}
\langle\F_{\V}(w), a\rangle=\langle w, \F_{\U}^{-1}(a)\rangle,\quad
w\in  \D'_{\V,\S},\; a\in \mathcal S(\ind).
\end{equation}
The Fourier transforms of elements of $\D'_{\U,\S}, \D'_{\V,\S}$ can be characterised by the property that, for example, for $w\in \D'_{\U,\S}$, there is $N>0$  and $C>0$ such that
$$|\F_{\U}w(\xi) | \leq C (1+|\lambda_{\xi}|)^{N},\quad \textrm{ for all }\;\xi\in\ind.$$

\item[(iii)]  $\U$-- and $\V$--convolutions can be extended by the same formula:
$$
f\s_{\U}g:= \sum_{\xi\in\ind}\widehat{f}(\xi) \widehat{g}(\xi) u_{\xi}=\sum_{\xi\in\ind}(f, v_\xi) (g, v_\xi) u_{\xi}
$$
for example, for all $f\in \D'_{\U,\S}$ and $g\in\C^{\infty}_{\U, \S}$. It is well-defined since the series converges in view of properties from (i) above and assumption \eqref{EQ:asl}.
By commutativity that spaces for $f$ and $g$ can be swapped.
Similarly,
$$
h\s_{\V}j:= \sum_{\xi\in\ind}\widehat{h}_{\ast}(\xi) \widehat{j}_{\ast}(\xi) v_{\xi}=\sum_{\xi\in\ind}(h, u_\xi) (j, u_\xi) v_{\xi}
$$
for each $h\in \D'_{\V,\S}$, $j\in\C^{\infty}_{\V, \S}$.
\end{itemize}

The space $\C^{\infty}_{\U, \S}$ can be also described in terms of the operator $\L$ in \eqref{EQ-1.1}. Namely, we have

\begin{equation}\label{EQ:Cl}
\C^{\infty}_{\U,\S}=\bigcap_{k\in\mathbb N_{0}}\Dom(\L^{k}),
\end{equation}
where $$\Dom(\L^{k}):=\{f\in\Sp: \,\, \L^{i}f\in\Sp, \, i=2, ... , k-1 \},$$
and similarly
$$\C^{\infty}_{\V,\S}=\bigcap_{k\in\mathbb N_{0}}\Dom((\Ls)^{k}),$$ where $$\Dom((\Ls)^{k}):=\{g\in\Sp: \,\, (\Ls)^{i}g\in\Sp, \, i=2, ... , k-1 \}.$$

Let $\mathcal S(\ind)$ denote the space of rapidly decaying
functions $\varphi:\ind\rightarrow\mathbb C$. That is,
$\varphi\in\mathcal S(\ind)$ if for any $M<\infty$ there
exists a constant $C_{\varphi, M}$ such that
\begin{equation}\label{EQ:ssp}
|\varphi(\xi)|\leq C_{\varphi, M}(1+|\lambda_{\xi}|)^{-M}
\end{equation}
holds for all $\xi\in\ind$.
The topology on $\mathcal
S(\ind)$ is given by the seminorms $p_{k}$, where
$k\in\mathbb N_{0}$ and $$p_{k}(\varphi):=\sup_{\xi\in\ind}(1+|\lambda_{\xi}|)^{k}|\varphi(\xi)|.$$
Continuous anti-linear functionals on $\mathcal S(\ind)$ are of
the form
$$
\varphi\mapsto\langle u, \varphi\rangle:=\sum_{\xi\in\ind}u(\xi)\overline{\varphi(\xi)},
$$
where functions $u:\ind \rightarrow \mathbb C$ grow at most
polynomially at infinity, i.e. there exist constants $M<\infty$
and $C_{u, M}$ such that
$$
|u(\xi)|\leq C_{u, M}(1+|\lambda_{\xi}|)^{M}
$$
holds for all $\xi\in\ind$. Such distributions $u:\ind
\rightarrow \mathbb C$ form the space of distributions which we denote by
$\mathcal S'(\ind)$, with the distributional duality (as a Gelfand triple) extending the inner product on $\ell^{2}(\ind)$.


Summarising the above definitions and discussion, we record the basic properties of the Fourier transforms  as follows:
\begin{prop}\label{LEM: FTinS}
The $\U$-Fourier transform
$\mathcal F_{\U}$ is a bijective homeomorphism from $\C^{\infty}_{\U, \S}$ to $\mathcal S(\ind)$.
Its inverse  $$\mathcal F_{\U}^{-1}: \mathcal S(\ind)
\rightarrow \C^{\infty}_{\U, \S}$$ is given by
\begin{equation}
\label{InvFourierTr} \mathcal F^{-1}_{\U}h=\sum_{\xi\in\ind}h(\xi)u_{\xi},\quad h\in\mathcal S(\ind),
\end{equation}
so that the Fourier inversion formula becomes
\begin{equation}
\label{InvFourierTr0}
f=\sum_{\xi\in\ind}\widehat{f}(\xi)u_{\xi}
\quad \textrm{ for all } f\in \C^{\infty}_{\U, \S}.
\end{equation}
Similarly,  $\mathcal F_{\V}:\C^{\infty}_{\V, \S}\to \mathcal S(\ind)$
is a bijective homeomorphism and its inverse
$$
\mathcal F_{\V}^{-1}: \mathcal S(\ind)\rightarrow
\C^{\infty}_{\V, \S}
$$ is given by
\begin{equation}
\label{ConjInvFourierTr}
\mathcal F^{-1}_{\V}h:=\sum_{\xi\in\ind}h(\xi)v_{\xi}, \quad h\in\mathcal S(\ind),
\end{equation}
so that the conjugate Fourier inversion formula becomes
\begin{equation}
\label{ConjInvFourierTr0}
f=\sum_{\xi\in\ind}\widehat{f}_{\ast}(\xi)v_{\xi}\quad \textrm{ for all } f\in \C^{\infty}_{\V, \S}.
\end{equation}
By \eqref{EQ: F_vd} the Fourier transforms extend to linear continuous mappings
$ \F_{\U}:\D'_{\U, \S}\to \mathcal S'(\ind)$ and
$ \F_{\V}:\D'_{\V, \S}\to \mathcal S'(\ind)$.
\end{prop}
The proof is straightforward.

\medskip

Let us formulate the properties of the $\U$- and $\V$-convolutions:

\begin{prop}\label{ConvProp}
For any $f\in \D'_{\U,\S}, g\in\C^{\infty}_{\U, \S}$, $h\in \D'_{\V,\S}, j\in\C^{\infty}_{\V, \S}$ we have
$$
\widehat{f\s_{\U} g}=\widehat{f}\,\widehat{g}, \,\,\,\, \widehat{h\s_{\V} j}_{\ast}=\widehat{h}_{\ast}\,\widehat{j}_{\ast}.
$$
The convolutions are commutative and associative.
If $g\in\C^{\infty}_{\U, \S}$ then for all
$f\in \D'_{\U,\S}$ we have
\begin{equation}\label{EQ:conv1}
f\s_{\U} g\in\C^{\infty}_{\U, \S}.
\end{equation}
\end{prop}
\begin{proof} Since the first part of the statement is proving in the same way as analogous one from Proposition \ref{PR: ConvProp}, we will show only the property \eqref{EQ:conv1} which follows if we observe that for all $k\in\mathbb N_{0}$ the series
$$
\sum_{\xi\in\ind}\widehat{f}(\xi) \widehat{g}(\xi) \lambda_{\xi}^{k}u_{\xi}
$$
converges since $\widehat{g}\in\mathcal S(\ind)$.
\end{proof}

\begin{prop} \label{PR: appl-1}
If $\L:\Sp\to\Sp$ is associated to a pair $(\U, \S)$ then we have
$$
\L(f\s_{\U}g)=(\L f)\s_{\U} g=f\s_{\U} (\L g)
$$
for any $f, g\in\C^{\infty}_{\U, \S}$.
\end{prop}
\begin{proof}
The proof is valid since the equalities
$$
\F_{\U}(\L(f\s_{\U}g))(\xi)=\lambda_{\xi}\widehat{f}(\xi)\widehat{g}(\xi)
$$
and
$$
\F_{\U}((\L f)\s_{\U} g)(\xi)=\F_{\U}(\L f)(\xi)\widehat{g}(\xi)=\lambda_{\xi}\widehat{f}(\xi)\widehat{g}(\xi)
$$
are true for all $\xi\in\ind$.
\end{proof}

As a small application, let us write the resolvent of the operator $\L$ in terms of the convolution.

\begin{thm}\label{TH: apll-2}
Let $\L:\Sp\to\Sp$ be an operator associated to a pair $(\U, \S)$. Then
the resolvent of the operator $\L$ is given by the formula
$$
\mathcal R(\lambda)f:=(\L-\lambda I)^{-1}f=g_{\lambda}\s_{\U}f,\quad \lambda\not\in\S,
$$
where $I$ is an identity operator in $\Sp$ and
$$
g_{\lambda}=\sum_{\xi\in\ind}\frac{1}{\lambda_{\xi}-\lambda}u_{\xi}.
$$
\end{thm}
\begin{proof} Begin by calculating the following series
\begin{align*}
g_{\lambda}\s_{\U}f&=\sum_{\xi\in\ind}\frac{1}{\lambda_{\xi}-\lambda} \widehat{f}(\xi) u_{\xi}\\
&=\sum_{\xi\in\ind}\widehat{f}(\xi)(\L-\lambda I)^{-1}u_{\xi}\\
&=(\L-\lambda I)^{-1}\left(\sum_{\xi\in\ind}\widehat{f}(\xi)u_{\xi}\right)\\
&=(\L-\lambda I)^{-1}f\\
&=\mathcal R(\lambda)f,
\end{align*}
where we used the continuity of the resolvent.
Now the theorem is proved.
\end{proof}

\section{Examples}

We give an example considered in \cite{RT16} that can be also considered as an extension setting in an appropriate sense of the toroidal calculus studied
in \cite{Ruzhansky-Turunen-JFAA-torus}.

Let the operator ${\rm O}_{h}^{(1)}:L^{2}(0,1)\to L^{2}(0,1)$ be given by the action
$$
{\rm O}_{h}^{(1)}:= -i\frac{d}{d x},
$$
where $h>0$, on the domain $(0, 1)$ with the boundary condition $h y(0)=y(1).$
In the case $h=1$ we have ${\rm O}_{1}^{(1)}$ with
periodic boundary conditions, and
the systems $\U$ and $\V$ of eigenfunctions of ${\rm O}_{1}^{(1)}$ and its adjoint
${{\rm O}_{1}^{(1)}}^{*}$ coincide, and are given by
$$
\U=\V=\{u_{j}(x)=e^{ 2\pi i x j },\,\, j\in \mathbb{Z}\}.
$$
This leads to the setting of the classical Fourier analysis on the circle which can be viewed as the interval $(0,1)$ with periodic boundary conditions.
The corresponding pseudo-differential calculus
was consistently developed in \cite{Ruzhansky-Turunen-JFAA-torus} building on
previous observations in the works by
Agranovich \cite{agran, agran2} and others.

For $h\not=1$, the operator ${\rm O}_{h}^{(1)}$  is not
self-adjoint. The spectral properties of ${\rm
O}_{h}^{(1)}$ are well-known (see Titchmarsh \cite{titc} and Cartwright \cite{cart}),
the spectrum of ${\rm O}_{h}^{(1)}$ is discrete and is given by
$\lambda_{j}=-i\ln h+2j\pi, \ j\in \mathbb{Z}.$
The corresponding bi-orthogonal families of eigenfunctions of ${\rm O}_{h}^{(1)}$ and its
adjoint are given by
$$
\U=\{u_{j}(x)=h^{x}e^{ 2\pi i x j },\,\, j\in \mathbb{Z}\}
$$
and
$$
\V=\{v_{j}(x)=h^{-x} e^{2\pi i x j },\,\, j\in \mathbb{Z}\},
$$
respectively. They form Riesz bases, and ${\rm O}_{h}^{(1)}$ is the operator associated to the pair $\U$ and  $\S=\{\lambda_{j}=-i\ln h+2j\pi\}_{j\in \mathbb{Z}}$.

Since $\mathbb N$ denoted an arbitrary discrete set before, all the previous constructions work with $\mathbb Z$ instead of $\mathbb N$.

Formally, we can write
\begin{equation}
\label{CONV} (f\s_{\U}
g)(x)=\int\int F(x,y,z)f(y)g(z)dydz,
\end{equation}
where
$$
F(x,y,z)=\sum_{\xi\in\ind}u_{\xi}(x) \ \overline{v_{\xi}(y)}
\ \overline{v_{\xi}(z)}.
$$
Here integrals \eqref{CONV} and the last series should be understood in the sense of distributions.
In the case $h=1$, it can be shown that
$F(x,y,z)=\delta(x-y-z)$, see \cite{Ruzhansky-Turunen-JFAA-torus}.

For any $h>0$, it can be shown that the $\U$--convolution
coincides with Kanguzhin's convolution that was studied in
\cite{Kanguzhin_Tokmagambetov} and \cite{Kanguzhin_Tokmagambetov_Tulenov}:

\begin{prop}\label{prop:Kc}
Let $\Sp=L^{2}(0,1)$, $\U=\{u_{j}(x)=h^{x}e^{ 2\pi i x j },\,\, j\in \mathbb{Z}\}$, and
$\S=\{\lambda_{j}=-i\ln h+2j\pi\}_{j\in \mathbb{Z}}$. Then the operator $\L:L^{2}(0,1) \to L^{2}(0,1)$
associated to the pair $(\U, \S)$ coincides with ${\rm O}_{h}^{(1)}$.
The corresponding $\U$-convolution can be written in the integral form:
$$
(f\star_{\U}
g)(x)=\int^{x}_{0}f(x-t)g(t)dt+\frac{1}{h}\int^{1}_{x}f(1+x-t)g(t)dt.
$$
\end{prop}
In particular, when $h=1$, we obtain
$$
(f\star_{\U}
g)(x)=\int_{0}^{1}f(x-t)g(t)dt.
$$
is the usual convolution on the circle.

\begin{proof}[Proof of Proposition \ref{prop:Kc}]
Let us denote
$$K(f,g)(x):=\int^{x}_{0}f(x-t)g(t)dt+\frac{1}{h}\int^{1}_{x}f(1+x-t)g(t)dt.
$$
Then we can calculate
\begin{align*}
&\F_{\U}(K(f,g))(\xi) \\ =&\int_{0}^{1}\int^{x}_{0}f(x-t)g(t)h^{-x} e^{-2\pi i x \xi }dtdx\\
&+\frac{1}{h}\int_{0}^{1}\int^{1}_{x}f(1+x-t)g(t)h^{-x} e^{-2\pi i x \xi }dtdx \\
=&\int_{0}^{1}\left[\int^{1}_{t}f(x-t)h^{-x} e^{-2\pi i x \xi }dx\right]g(t)dt\\
&+\int_{0}^{1}\left[\int^{t}_{0}f(1+x-t)h^{-(1+x)} e^{-2\pi i (1+x) \xi }dx\right]g(t)dt
\\
=&\int_{0}^{1}\left[\int^{1}_{t}f(x-t)h^{-(x-t)} e^{-2\pi i (x-t) \xi }dx\right]g(t)h^{-t} e^{-2\pi i t \xi }dt\\
&+\int_{0}^{1}\left[\int^{t}_{0}f(1+x-t)h^{-(1+x-t)} e^{-2\pi i (1+x-t) \xi }dx\right]g(t)h^{-t} e^{-2\pi i t \xi }dt
\end{align*}
\begin{align*}
\,\,\,\,\,\,\,\,\,\,\,\,\,\,\,\,\,\,\,\,\,\,\,\,\,\,\,\,\,\,\,\,\,\,
\,\,\,\,\,\,\,\,\,\,\,\,\,\,\,\,\,\,\,\,\,
=&\int_{0}^{1}\left[\int^{1-t}_{0}f(z)h^{-z} e^{-2\pi i z \xi }dz\right]g(t)h^{-t} e^{-2\pi i t \xi }dt\\
&+\int_{0}^{1}\left[\int^{1-t}_{1}f(z)h^{-z} e^{-2\pi i z \xi }dx\right]g(t)h^{-t} e^{-2\pi i t \xi }dt
\end{align*}
\begin{align*}
\,\,\,\,\,\,\,\,\,\,\,\,\,\,\,\,\,\,\,\,\,\,\,\,\,\,\,\,\,\,\,\,\,\,
\,\,\,\,\,\,\,\,\,\,\,\,\,\,\,\,
=&\int_{0}^{1}\left[\int^{1}_{0}f(z)h^{-z} e^{-2\pi i z \xi }dz\right]g(t)h^{-t} e^{-2\pi i t \xi }dt\\
=&\widehat{f}(\xi)\,\widehat{g}(\xi).
\end{align*}
Consequently, by Theorem \ref{PR: ConvProp}, we obtain that
$K(f,g)=f*_{\U}g$.
\end{proof}


\section{Further discussion}

We note that in the case when we are given an operator $\L:\Sp\to\Sp$ for which the eigenfunctions do not make a basis of $\Sp$, other notions of convolutions are possible,
still satisfying an analogue of Proposition \ref{PR: appl-1}.
Here we can set up a convolution using the characterisation of the space of test functions given by \eqref{EQ:Cl}.

\begin{defi}\label{DEF: L-oper} Let $\L:\Sp\to\Sp$ be a linear densely defined operator in $\Sp$.
Denote $\Dom(\L^{\infty}):=\bigcap_{k\in\mathbb N_{0}}\Dom(\L^{k})$ with $\Dom(\L^{k}):=\{f: \,\, \L^{i}f\in\Sp, \, i=2, ... , k-1 \}.$
We say that  a bilinear associative and commutative operation $*_{\L}$ is an $\L$-convolution if
for any $f,g\in\Dom(\L^{\infty})$
we have
$$
\L(f\s_{\L}g)=(\L f)\s_{\L}g=f\s_{\L}(\L g).
$$
\end{defi}

Proposition \ref{PR: appl-1} implies that $\U$-convolution is a special case of $\L$-convolutions:

\begin{cor}
Assume that $\L:\Sp\to\Sp$ is an operator associated with the pair $(\U, \S)$, where $\U$ is a Riesz basis in $\Sp$. Then the $\U$--convolution $\s_{\U}$ is an $\L$-convolution.
\end{cor}

We finally show that an $\L$-convolution does not have to be a $\U$-convolution for any choice of the set $\S$.

For this, let us consider an $\L$-convolution associated to the so-called Ionkin operator considered in \cite{I77}.
The Ionkin operator $\Y:\Sp\to\Sp$ is the operator in  $\Sp:=L^{2}(0, 1)$ generated by the differential expression
$$
-\frac{d^{2}}{d x^{2}}, \,\,\, x\in(0, 1),
$$
with the boundary conditions
$$
u(0)=0, \,\,\, u'(0)=u'(1).
$$
It has eigenvalues
$$
\lambda_{\xi}=(2\pi \xi)^{2},\; \xi\in\mathbb Z_{+},
$$
and an extended set of eigenfunctions
$$
u_{0}(x)=x, \,\,\, u_{2\xi-1}(x)=\sin(2\pi \xi x), \,\,\, u_{2\xi}(x)=x\cos(2\pi \xi x), \,\, \xi\in\mathbb N,
$$
which give a basis in $L^{2}(0, 1)$, which we denote by $\U$. The corresponding biorthogonal basis is given by
$$
v_{0}(x)=2, \,\,\, v_{2\xi-1}(x)=4(1-x)\sin(2\pi \xi x), \,\,\, v_{2\xi}(x)=4\cos(2\pi \xi x), \,\, \xi\in\mathbb N,
$$
for more details, see \cite{I77}. We consider the $\Y$--convolution (Ionkin--Kanguzhin's convolution) given by the formula
\begin{multline}
f\s_{\Y}g(x):=\frac{1}{2}\int_{x}^{1}f(1+x-t)g(t)dt\\
+\frac{1}{2}\int_{1-x}^{1}f(x-1+t)g(t)dt+\int_{0}^{x}f(x-t)g(t)dt\\
-\frac{1}{2}\int_{0}^{1-x}f(1-x-t)g(t)dt+\frac{1}{2}\int_{0}^{x}f(1+t-x)g(t)dt.
\end{multline}
This is a $\Y$-convolution  in the sense of
Definition \ref{DEF: L-oper}, namely, it satisfies
$$
\Y(f\s_{\Y}g)=(\Y f)\s_{\Y}g=f\s_{\Y}(\Y g),
$$
see \cite{KT15}.
For the collection
$$
\U:=\{u_{\xi}: \,\,\, u_{0}(x)=x, \, u_{2\xi-1}(x)=\sin(2\pi \xi x), \,\, u_{2\xi}(x)=x\cos(2\pi \xi x), \,\, \xi\in\mathbb N\},
$$
it can be readily checked that the corresponding $\U$-Fourier transform satisfies
$$
\widehat{f\s_{\Y}g}(0)=\widehat{f}(0)\widehat{g}(0),
$$
$$
\widehat{f\s_{\Y}g}(2\xi)=\widehat{f}(2\xi)\widehat{g}(2\xi),
$$
$$
\widehat{f\s_{\Y}g}(2\xi-1)=\widehat{f}(2\xi-1)\widehat{g}(2\xi)+\widehat{f}(2\xi)\widehat{g}(2\xi)
+\widehat{f}(2\xi)\widehat{g}(2\xi-1), \,\, \xi\in\mathbb N.
$$
Therefore, by Theorem \ref{PR: ConvProp}, the $\Y$--convolution (Ionkin--Kanguzhin convolution) does not coincide with
the
$\U$--convolution for any choice of numbers $\S$.

\section{Acknowledgment}

The authors thank Professor Baltabek Kanguzhin for interesting discussions, and the referee for useful remarks.

\end{document}